\documentclass[11pt]{article}
\usepackage{mathtools}
\usepackage{graphicx}
\usepackage{tikz}
\usepackage[style=alphabetic]{biblatex}
\usepackage{amsmath} 
\usepackage{amsthm}
\usepackage{amsfonts} 
\usepackage{accents}
\usepackage[a4paper, portrait, margin=1in]{geometry}
\usepackage{amssymb}
\usepackage{wasysym}
\usepackage{comment}
\usepackage{hyperref}
\usepackage{float}
\usepackage{enumerate}
\usepackage{multicol}
\usepackage{hyperref}
\usepackage{moresize}

\usetikzlibrary{arrows.meta, calc, quotes, tikzmark}

\bibliography{references.bib}

\newtheorem{theorem}{Theorem}[section]
\newtheorem{lemma}[theorem]{Lemma}

\newtheorem{corollary}[theorem]{Corollary}

\theoremstyle{definition}
\newtheorem{remark}{Remark}
\newtheorem{definition}[theorem]{Definition}

\numberwithin{equation}{section}

\hypersetup{
    colorlinks=true,
    linkcolor=blue,
    filecolor=magenta,   
    citecolor = magenta,
    urlcolor=cyan,
    pdftitle={Overleaf Example},
    pdfpagemode=FullScreen,
    }

\newcommand{\ZZ}{\mathbb{Z}}

\newcommand{\logphi}[1]{\mathrm{log}_{\varphi}\left( #1 \right)}

\newcommand{\logPhi}[1]{\mathrm{log}_{\phi_{k}}\left( #1 \right)}

\newcommand{\floorof}[1]{\left\lfloor #1 \right\rfloor}

\newcommand{\set}[1]{\{ #1 \}}

\newcommand\blfootnote[1]{%
  \begingroup
  \renewcommand\thefootnote{}\footnote{#1}%
  \addtocounter{footnote}{-1}%
  \endgroup
}

\title{Variants of Conway Checkers and $k$-nacci Jumping}
\date{}
\author{Glenn Bruda, Joseph Cooper, Kareem Jaber, \\ Raul Marquez and Steven J. Miller}

\begin{document}

\maketitle

\setcounter{section}{0}

\begin{abstract}

Conway Checkers is a game played with a checker placed in each square of the lower half of an infinite checkerboard. Pieces move by jumping over an adjacent checker, removing the checker jumped over. Conway showed that it is not possible to reach row 5 in finitely many moves by weighting each cell in the board by powers of the golden ratio such that no move increases the total weight.

Other authors have considered the game played on many different boards, including generalising the standard game to higher dimensions. We work on a board of arbitrary dimension, where we allow a cell to hold multiple checkers and begin with $m$ checkers on each cell. We derive an upper bound and a constructive lower bound on the height that can be reached, such that the upper bound is almost always equal to the lower bound.

We also consider the more general case where instead of jumping over 1 checker, each checker moves by jumping over $k$ checkers, and again show the maximum height reachable lies within bounds that are almost always equal.

\end{abstract}

\renewcommand\abstractname{Dedication}
\begin{abstract}
This paper is dedicated with thanks to Peter G. Anderson, Marjorie Bicknell-Johnson and William Webb. In a similar fashion to the pagoda functions crucial to this paper, their tireless effort and leadership has allowed the journal and association to take only an upward trajectory for decades, and it is a great pleasure to acknowledge their service and mentorship.
\end{abstract}

\tableofcontents

\blfootnote{
\newline
Glenn Bruda: \href{mailto:glenn.bruda@ufl.edu}{glenn.bruda@ufl.edu} (University of Florida),

\noindent Joseph Cooper: \href{mailto:jc2407@cam.ac.uk}{jc2407@cam.ac.uk} (University of Cambridge),

\noindent Kareem Jaber: \href{mailto:kj5388@princeton.edu}{kj5388@princeton.edu} (Princeton University), 

\noindent Raul Marqeuz:
\href{mailto:raul.marquez02@utrgv.edu}{raul.marquez02@utrgv.edu} (University of Texas Rio Grande Valley),

\noindent Steven J. Miller:
\href{mailto:sjm1@williams.edu}{sjm1@williams.edu} (Williams College).
}

\newpage

\section{Introduction}

First introduced by John Conway and summarised well in \cite{winningways}, Conway Checkers, or Conway Soldiers, is an interesting mathematical game with a beautiful and possibly counterintuitive solution.

To set up the game, the player places a checker on each square below some initial line on an infinite checkerboard. The game is then to try and move a checker as high up the board as possible, where a checker moves by jumping over another checker it shares an edge with, removing the piece jumped over from the board.

Conway's solution involved weighting each square of the board by powers of the golden ratio $\varphi$, such that whatever move the player makes, the total weight of all the squares with a checker on never increases. We follow terminology introduced by Conway and call this weight function a pagoda function. Using this, it is possible to show that, in finitely many moves, it is not possible to reach the 5\textsuperscript{th} row above the initial line.

The use of these pagoda functions is very common in the study of games similar to Conway Checkers, for example Peg Solitaire. For a foundational text, see again \cite{winningways}.

In \cite{row5}, the authors show a construction of how to reach row 5 when infinitely many moves are allowed to be made, as long as you consider each infinite amount of moves to occur in a finite time period, so it is possible to perform multiple sequences of infinite moves consecutively.

Many generalisations of this game have been considered. Some involve changing the moves allowed. In \cite{minsize}, the authors consider the game except pieces move as in regular checkers, as introduced in \cite{hexagonal}, jumping only diagonally. They also consider the case where pieces can move either diagonally or orthogonally.

Also in \cite{minsize}, the authors use variations on Conway's pagoda function, as well as some constructions, to find the minimal number of checkers that are required to reach certain rows in various different games, including the two mentioned above. This is also done in \cite{desert}, where the author uses an entirely different pagoda function to Conway's, weighting each square on an arbitrarily large finite board by the $(n-t)$\textsuperscript{th} Fibonacci number, where $t$ is the square's taxicab distance from a target square and $n$ is large enough.

Variations on the board the game is played on have also been considered. In \cite{hexagonal}, the authors introduce the game played on a board with hexagonal cells, so that each piece has six choices of directions to jump. This is built upon by \cite{minsize}, who again prove properties about the minimal number of checkers to reach the maximum row.

\cite{gunpaper} provides constructions on how to reach the highest possible rows on a number of different initial boards, including a board where all but the first quadrant is initially filled, and the same board rotated by an eighth of a turn.

In \cite{graphgeneralisation}, the author generalises the game to be played on an arbitrary connected, countable and locally finite graph. They show that on such a graph, a point is unreachable from some initial condition where all checkers are at least some distance $k$ away if the size of the set of points distance $n$ from that point is $O(\varphi^{\varepsilon n})$ for an $\varepsilon < 1$.

In \cite{ddimpaper}, the game is generalised from 2-dimensions to $d$-dimensions. The authors show that it is always possible to bring a checker to row $3d-2$ and never row $3d-1$. We generalise this result further.

We consider a variant of the the game played in $d$-dimensions, as in \cite{ddimpaper}, where each square can hold any number of checkers, and the initial board position has $m$ checkers occupying the squares instead of 1. We also consider the further generalisation, still with $m$ checkers on each square, where instead of jumping over 1 checker, each checker moves by jumping over $k-1$ checkers in a straight line, for some $k \geq 2$.  We call this game the Conway $(m,k,d)$-game and prove the following theorem.

\begin{theorem} \label{mainresult}
    In the Conway $(m,k,d)$-game, with $m>1$, the maximum row attainable $n_M$ satisfies
    \begin{equation} \label{firstresult}
    \floorof{\logPhi{m} + \logPhi{\frac{(\phi_{k}+1)^{d-1}}{(\phi_{k}-1)^d}}} \ \leq \ n_M \ \leq \ \floorof{\logPhi{m} + \logPhi{\frac{(\phi_{k}+1)^{d-1}}{(\phi_{k}-1)^d}}} + 1,
    \end{equation}
    where $\phi_{k}$ is the $k$-nacci constant as defined in Lemma \ref{knacciconstant}. In particular, we have that for almost all values of $m$, the upper bound is attainable. When $k=2$ we have that for all $m \in \mathbb{N}$ 
    \begin{equation} \label{1point1}
    \floorof{\logphi{m}} + 3d-2 \ \leq \ n_M \ \leq \ \floorof{\logphi{m}} + 3d-1.
    \end{equation}
    Again, this upper bound is attainable for almost all $m$.
\end{theorem}
The result \eqref{1point1} follows for $m>1$ from \eqref{firstresult} and for $m=1$ from results in \cite{ddimpaper}. For details on why we fail to prove \eqref{firstresult} in the $m=1$ case see Remark \ref{remark4}.

The bound in \eqref{firstresult} is logarithmic in $m$, and linear in $d$. Matching with intuition, this tells us that increasing the dimension gives much more power to move up the board than increasing the number of checkers. By Remark \ref{knaccitoinf}, these bounds behave like $\log_{2}(m) + (d-1)\log_2(3)$ as $k$ grows.

Since the game in $d$-dimensions is very complicated, it is difficult to directly prove strong results about it. Instead we first prove results in the 1-dimensional case concerning how high a row can be reached, as well as how many checkers can be moved onto the first row. We then use these results to obtain bounds in the $d$-dimensional case.

We also use these results in 1-dimension to prove results about the maximum amount of checkers that can be moved onto a single square in higher dimensions, and show this asymptotically matches an upper bound for large $m$.

\section{Definitions and Preliminaries}

The original rules of Conway Checkers are as follows.

\begin{itemize}
    \item On an infinite checkerboard, start by placing a checker on each square below a certain arbitrary line. We call the top row of checkers row 0.
    \item On each move, a checker jumps over an orthogonally adjacent checker and lands on the other side of it. The checker it jumped over is then removed.
\end{itemize}

\noindent When generalising to higher dimensions $d$, we consider the game as being played on $\ZZ^d$, with each point being a cell of the board.

\begin{definition} \label{dfunc}
    Given a point $T \in \ZZ^d$, define the distance function $d \, (\cdot, T): \ZZ^d \to \ZZ_{\geq 0}$ by
    \begin{equation}
        d(x, T) \ = \ \sum_{i=1}^{d} |x_i-T_i|,
    \end{equation}
    the taxicab distance between $x$ and $T$.
\end{definition}

\begin{definition}
    Given a function $f: \ZZ^d \to \mathbb{R}$, the \textit{energy} of a boardstate $B$ with respect to $f$ is the sum of the function applied to each checker's square. For example, if $f = 1$ then the energy of the boardstate is just the number of checkers on the board.
\end{definition}

\begin{definition}
    A function $p: \ZZ^d \to \mathbb{R}$ is called a \textit{pagoda function} if the energy of the board with respect to $p$ is non-increasing for any possible move.
\end{definition}

\begin{theorem}
We have the following two conditions on a boardstate being unattainable in finite time (where all energies are with respect to an arbitrary pagoda function).

\begin{enumerate}[(i)] \label{fundamentaltheoremofconwaycheckers}
    \item For a target boardstate $B$, if the energy of $B$ is greater than the initial energy of the board, then that boardstate is not attainable.
    
    \item If there are infinitely many checkers that are present in the initial boardstate and not present in the target state, then it is not possible to reach $B$ in finitely many moves.
\end{enumerate}

\end{theorem}

\begin{proof}

To prove $(i)$ we note that the board's energy is non-increasing for all possible moves, hence no sequence of moves can lead to a boardstate with greater energy. To show $(ii)$ we note that if there are infinitely many such checkers, then infinitely many checkers need to be moved. Since the rules only permit the moving of one checker at a time, it is not possible to reach $B_0$ in finite time.
\end{proof}
\begin{remark}
    Note that these conditions give no indication as to whether a state is attainable, only that it is not possible to reach.
\end{remark}

\begin{definition} \label{knaccidef}
    We define the $k$-nacci sequence $F_k(n)$ by the recurrence
    \begin{equation}
    a_{n+k}  \ = \  a_{n} + a_{n+1} + \cdots + a_{n+k-1}
    \end{equation}
    and initial conditions
    \begin{equation}
        F_k(0)  \ = \  F_k(1)  \ = \  \cdots  \ = \  F_k(k-2)  \ = \  0, \; F_k(k-1)  \ = \  1.
    \end{equation}
\end{definition}

\begin{lemma} \label{knacciconstant}
    The characteristic polynomial of the $k$-nacci sequence,
    \begin{equation}
        1+x+\cdots+x^{k-1}-x^k
    \end{equation}
    has only one root $\phi_{k}$ with $|\phi_{k}| \geq 1$. In particular, this root is always real and lies in the interval $(1,2)$. We call this the $k$-nacci constant. In the case $k=2$, we denote this by $\varphi$, the golden ratio, equal to $\frac{1+\sqrt{5}}{2}$. We also have that all the roots of this polynomial are distinct.
\end{lemma}
\begin{proof}
    This follows by Rouch\'{e}'s theorem. For a proof see, for example, \cite{unitcircle}.
\end{proof}

\begin{remark} \label{knaccitoinf}
    We also have that $\phi_k \to 2$ as $k \to \infty$. This is shown in \cite{unitcircle2}, which also gives an alternative proof to Lemma \ref{knacciconstant}.
\end{remark}

\begin{corollary} \label{asymptoticfk}
    We have $F_k(n) \sim c\phi_{k}^n$, where $c$ is real and positive.
\end{corollary}
\begin{proof}
    Since the roots of the characteristic polynomial are all distinct, the solution to the linear recurrence can be written as a sum of powers of the roots of its characteristic equation. We then have
    \begin{equation}
        F_k(n)  \ = \  c\phi_{k}^n + A_1\rho_1^n + \cdots + A_{k-1}\rho_{k-1}^n
    \end{equation}
    for a collection of $\rho_i$ that lie inside the unit circle. It then clearly follows that
    \begin{equation}
        \lim_{n\to \infty}\frac{F_k(n)}{c\phi_{k}^n}  \ = \  1,
    \end{equation}
    so $F_k(n) \sim c\phi_{k}^n$ for some $c$. From the initial conditions in Definition \ref{knaccidef}, we have that $F_k(n)$ is real and non-negative for all $n \in \mathbb{N}$ and $F_k(n) \to \infty$ as $n \to \infty$, so $c$ must be real and positive.
\end{proof}

\begin{lemma} \label{knacciplusone}
    The solution to the recurrence
    \begin{equation}
        a_{n+k}  \ = \  a_{n} + \cdots + a_{n+k-1} + 1
    \end{equation}
    with initial conditions
    \begin{equation}
        (a_0,a_1,a_2,\dots,a_{k-1})  \ = \  (1,2,4, \dots, 2^{k-1})
    \end{equation}
    is 
    \begin{equation}
        a_i  \ = \  \sum_{j=0}^{i} F_k(k-1+j).
    \end{equation}
\end{lemma}
\begin{proof}
    This follows from verifying the given solution and using uniqueness.
\end{proof}

\begin{definition}
    We define the value $n_M$ to be the maximum row attainable,
    \begin{equation}
        n_M \ := \  \mathrm{max}\left( 
\set{n: \text{\normalfont row $n$ is attainable in finitely many moves}} \right).
    \end{equation}
\end{definition}

\section{Conway's Solution}

Since we intend to use similar ideas later, we show, by methods of Conway illustrated in \cite{winningways}, that it is impossible to reach row 5 in finitely many moves.\bigskip

On a board $\ZZ^2$, fix a target square $T$ above the starting line. We then weight each square of the board by $\alpha^d$ for some $\alpha \in (0,1)$ that we fix later, where $d$ is the square's taxicab distance from $T$ as in Definition \ref{dfunc} (so $T$ has weight 1).

\begin{table}[ht]
    \centering
    \begin{tabular}{c|c|c|c|c|c|c|c|c}
    \hline
        $\cdots$ & $\alpha^3$ & $\alpha^2$ & $\alpha^1$ & $\alpha^0$ & $\alpha^1$ & $\alpha^2$ & $\alpha^3$ & $\cdots$ \\
    \hline
        $\cdots$ & $\alpha^4$ & $\alpha^3$ & $\alpha^2$ & $\alpha^1$ & $\alpha^2$ & $\alpha^3$ & $\alpha^4$ & $\cdots$ \\
    \hline
        $\cdots$ & $\alpha^5$ & $\alpha^4$ & $\alpha^3$ & $\alpha^2$ & $\alpha^3$ & $\alpha^4$ & $\alpha^5$ & $\cdots$ \\
    \hline
        $\vdots$ & $\vdots$ & $\vdots$ & $\vdots$ & $\vdots$ & $\vdots$ & $\vdots$ & $\vdots$ & $\vdots$ \\
    \hline
        $\cdots$ & $\alpha^{n+2}$ & $\alpha^{n+1}$ & $\alpha^n$ & $\alpha^{n-1}$ & $\alpha^n$ & $\alpha^{n+1}$ & $\alpha^{n+2}$ & $\cdots$ \\
    \hline\hline
        $\cdots$ & $\alpha^{n+3}$ & $\alpha^{n+2}$ & $\alpha^{n+1}$ & $\alpha^{n}$ & $\alpha^{n+1}$ & $\alpha^{n+2}$ & $\alpha^{n+3}$ & $\cdots$ \\
    \hline
        $\cdots$ & $\alpha^{n+4}$ & $\alpha^{n+3}$ & $\alpha^{n+2}$ & $\alpha^{n+1}$ & $\alpha^{n+2}$ & $\alpha^{n+3}$ & $\alpha^{n+4}$ & $\cdots$ \\
    \hline
        $\vdots$ & $\vdots$ & $\vdots$ & $\vdots$ & $\vdots$ & $\vdots$ & $\vdots$ & $\vdots$ & $\vdots$ \\
    \end{tabular}
    \caption{Weighting of the board, with row 0 below the double line.}
    \label{tab:boardweight}
\end{table}

The initial boardstate contains a checker on each square below some line, denoted with the double bar in Table \ref{tab:boardweight}. We call this row of checkers row 0. Hence we can count the total initial weight of the board by summing the weights of all squares with a checker. This gives

\begin{equation}
    E_0  \ = \  \frac{\alpha^n(1+\alpha)}{(1-\alpha)^2}.
\end{equation}

We want the energy of the board to be non-increasing under all possible moves. To do this, we consider a move towards the target square. This has change in energy 
\begin{equation}
\Delta E \  =  \ \alpha^k(1 - \alpha - \alpha^2).
\end{equation}
It makes sense to set this to equal zero since moving toward the target is intuitively the `best' possible move. We pick the value of $\alpha$ with $|\alpha| < 1$ so that the initial energy is finite, giving $\alpha = \frac{\sqrt{5}-1}{2} = 1/\varphi$ and $E_0 = \alpha^{n-5}$. Considering the change in energy under all other possible moves, both horizontal and vertical, it is easy to show that with this value of $\alpha$ the energy must always be non-increasing.\bigskip

Conway's result, that row 5 is not attainable in finitely many moves, is then an immediate consequence of Theorem \ref{fundamentaltheoremofconwaycheckers}. To see this, set $n=5$. Then we have $E_0 =1$, but the energy of just the target square $T$ is also 1. Hence $B$ must just be the square $T$, else $(i)$ is violated. Then row 5 is not attainable in finitely many moves by $(ii)$.

It is possible to explicitly find a sequence of moves that will allow row 4 to be reached, and in fact this only requires a minimum of 20 checkers. A possible starting position with 20 checkers can be found in \cite{winningways}.\bigskip

\section{Generalisations}

We now generalise the game. We make the following alterations.

\begin{itemize}
    \item Each square can now have any number of checkers occupying it.
    \item Each square that would usually start with a checker starts with $m$ checkers.
    \item Checkers now jump over $k-1$ pieces instead of just one, removing a checker from each square they jump over.
\end{itemize}

When we consider this variant on the board $\ZZ^d$, we call this the Conway $(m,k,d)$-game. We define moves as jumping over $k-1$ checkers instead of $k$ for notational convenience. Note that earlier the condition that each checker can only jump into an empty square was not in fact relevant for the solution, the game behaved the same with or without this condition.

\begin{table}[ht]
    \centering
    \begin{tabular}{|c|c|c|c|}
    \hline
        $1$ & $1$ & $1$ & $0$ \\
    \hline
    \end{tabular}
    \label{tab:movept1}
\end{table}

\begin{figure}[ht]
    \centering
    \begin{tabular}{|c|c|c|c|}
    \hline
        $0\tikzmark{e}$ & $0$ & $0$ & \tikzmark{f}$1$ \\
    \hline
    \end{tabular}
    \label{tab:movept2}

    \begin{tikzpicture}[overlay, remember picture, shorten >=.5pt, shorten <=.5pt, transform canvas={yshift=.9\baselineskip}]
    \draw [->] ({pic cs:e}) [bend left] to ({pic cs:f});
    \end{tikzpicture}
    \caption{A move, with $k=3$.}
\end{figure}
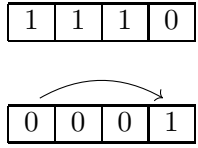

Note that the Conway $(1,2,2)$-game is just the standard game as defined earlier.\bigskip

To analyse this game, we define a new pagoda function. We construct this analogously to earlier, setting $p(x) = \alpha^{d(x,T)}$ and $\alpha$ to be the appropriate solution to

\begin{equation}
    1-\alpha-\cdots-\alpha^{k-1}-\alpha^k  \ = \  0.
\end{equation}
We choose $\alpha=1/\phi_{k}$, where $\phi_{k}$ is the $k$-nacci constant, as in Definition \ref{knacciconstant}. This ensures that the initial boardstate has finite energy.

\subsection{\texorpdfstring{k}{Lg}-nacci Jumping}

In the hopes of finding a lower bound on the height that can be attained in the Conway $(m,k,d)$-game, we first show a construction that allow us to reach row $n$. We define 

\begin{equation} \label{sidef}
S_i(n) \ := \ \sum_{j=0}^{i}F_k(n+k-2-j),
\end{equation}
recalling that $F_k$ denotes the $k$-nacci sequence with initial conditions as in Definition \ref{knaccidef}.

\begin{lemma} \label{nnaccijumping}
    Suppose that in one column, in row $-i$, for $i = 0,\dots, k-1$, we have $S_{k-i-1}(n)$ checkers. Then it is possible to reach the $n^{\mathrm{th}}$ row.
\end{lemma}

\begin{proof}
In Table \ref{tab:nnaccijumping} below, the first column gives the initial state of the board after expanding each $S_{k-i-1}$. The second column gives the state after jumping all checkers from the lowest square over the rest.

\begin{table}[ht]
    \centering
    \begin{tabular}{c||c}
        move 0 & move 1 \\
    \hline
        0 & 0 \\
    \hline
        0 & $F_k(n+k-2)$ \\
    \hline\hline
        $F_k(n+k-2) + \dots + F_k(n-1)$ & $F_k(n+k-3) + \dots + F_k(n-1)$ \\
    \hline
        $F_k(n+k-2) + \dots + F_k(n)$ & $F_k(n+k-3) + \dots + F_k(n)$ \\
    \hline
        \vdots & \vdots \\
    \hline
        $F_k(n+k-2)+F_k(n+k-3)$ & $F_k(n+k-3)$ \\
    \hline
        $F_k(n+k-2)$ & 0 \\
    \hline
    \end{tabular}
    \caption{First move.}
    \label{tab:nnaccijumping}
\end{table}

Then, since $F_k(n+k-2) = F_k(n+k-3)+\dots + F_k(n-2)$,
we are in the same state as before with $n$ replaced by $n-1$, and one row higher. Repeat this process, and it is possible to reach row $n$.

\end{proof}

\subsection{Solution in \texorpdfstring{$\ZZ$}{Lg}: \texorpdfstring{$k=2$}{Lg}}

When the game is played on $\ZZ$, it is intuitively clear that the only good moves are moves up the board. We show that this intuition is correct, by finding an upper bound and proving that the player can always reach this bound using only upward moves, first in the case $k=2$ and then in the more general case for any choice of $k$. This is a very important result, as our solution in $d$-dimensions depends on projecting down into 1-dimension and applying this result.\newline

To motivate the more general proof, we will show that with $k=2$, $d=1$ and $m>1$ arbitrary (so each checker jumps over 1 other) we have that the maximum row that can be reached is exactly $\floorof{\logphi{m}}+2$.\newline

To show this is the upper bound, the standard weighting of the board gives an initial energy of $E_0 = \frac{m\alpha^n}{1-\alpha}$, and we solve $E_0 < 1$ to obtain

\begin{equation}
    n_M \ \leq \ \floorof{\logphi{m}}+2.
\end{equation}

To obtain the lower bound, we will attempt to climb the board using $k$-nacci jumping. In this case, we have $S_0 = F(n)$, $S_1 = F(n)+F(n-1) = F(n+1)$, (writing $F$ instead of $F_2$ for the Fibonacci sequence), and $k$-nacci jumping becomes Fibonacci jumping. The process of reaching row $n$ is illustrated in Table \ref{tab:fibjumping}.

    \begin{table}[ht]
        \centering
        \begin{tabular}{|c||c||c||c||c||c||}
        \hline
            row & move 0 & move 1 & $\cdots$ & move $n-1$ & move $n$ \\
        \hline\hline
            $n$ & 0 & 0 & $\cdots$ & 0 & $F(1)$\\
        \hline
            $n-1$ & 0 & 0 & $\cdots$ & $F(2)$ & $F(0)$\\
        \hline
            $n-2$ & 0 & 0 & $\cdots$ & $F(1)$ & 0\\
        \hline
            $\vdots$ & $\vdots$ & $\vdots$ & $\vdots$ & $\vdots$ & $\vdots$\\
        \hline
            2 & 0 & 0 & $\cdots$ & 0 & 0\\
        \hline
            1 & 0 & $F(n)$ & $\cdots$ & 0 & 0\\
        \hline
        \hline
            0 & $F(n+1)$ & $F(n-1)$ & $\cdots$ & 0 & 0\\
        \hline
            -1 & $F(n)$ & 0 & $\cdots$ & 0 & 0\\
        \end{tabular}
        \caption{Fibonacci jumping.}
        \label{tab:fibjumping}
    \end{table}

Guided by Lemma \ref{nnaccijumping}, we aim to put $F(n+1)$ checkers in row 0 and $F(n)$ in row $-1$ in order to reach the $n$\textsuperscript{th} row. To each row we need to add the following amounts of checkers.

\begin{itemize}
    \item In row 0 we need to add $F(n+1)-m$ checkers.
    \item In row $-1$ we need to add $F(n)-m$, as well as the $F(n+1)-m$ required to be added to row 0. This is a total of $F(n+2)-2m$.
    \item In row $-2$ we need to add enough to jump into row 0, and enough to be jumped over into row $-1$. This is a total of $F(n+3)-3m$, however there are already $m$ checkers in this row, so we need to add $F(n+3)-4m$ checkers.
    \item Continue this down the board, with each square requiring enough checkers to jump into the square two above it, as well as enough checkers to be jumped over into the square directly above, minus the $m$ checkers that each square already contains.
\end{itemize}

\begin{table}[ht]
    \centering
    \begin{tabular}{|c||c|c|c|}
\hline
    row & move 0 & target & amount needed \\
\hline
\hline
    0 & $m$ & $F(n+1)$ & $F(n+1)-m$ \\
\hline
    $-1$ & $m$ & $F(n+2)-m$ & $F(n+2)-2m$ \\
\hline
    $-2$ & $m$ & $F(n+3)-3m$ & $F(n+3)-4m$ \\
\hline
    $-3$ & $m$ & $F(n+4)-6m$ & $F(n+4)-7m$ \\
\hline
    $-4$ & $m$ & $F(n+5)-11m$ & $F(n+5)-12m$ \\
\hline
    $-5$ & $m$ & $F(n+6)-19m$ & $F(n+6)-20m$ \\
\hline
    \vdots & \vdots & \vdots & \vdots\\
\end{tabular}
\caption{Checkers needed to be added in each square.}
\end{table}

Note that beyond row $-1$, the target amount is the sum of the amount needed to be added in the two squares above, since the square needs to jump into one and be jumped over into the other. The amount that needs to be added is then the target value minus $m$, since there are initially $m$ checkers in the square.

We see that in row $-i$, we need to add

\begin{equation}
    F(n+i+1)-a_im
\end{equation}
where $a_{i} = a_{i-1}+a_{i-2}+1$. By Lemma \ref{knacciplusone}, since the initial conditions match, this is

\begin{equation}
    F(n+i+1)-\left(\sum_{j=0}^{i}F(j+1)\right)m  \ = \ F(n+i+1)-(F(i+3)-1)m.
\end{equation}

If this were to be eventually non-positive for consecutive rows, we would have that beyond a certain point, there are enough checkers on the squares to support reaching the desired state, since no more need to be added. We could then go from this point upwards, adding the required amount of checkers to each square, and we would have that row $n$ could be reached. Note this only uses a finite section of the column.\newline

To this end, we apply Corollary \ref{asymptoticfk} to find a condition on this being eventually non-positive. We have
\begin{align}
    \notag &F(n+i+1) - \left( \sum_{j=0}^{i}F(j+1) \right)m \\ \notag
    \sim \ \ &c\varphi^{n+i+1}-cm\left(\varphi+ \cdots + \varphi^{i+1}\right) \\ \notag
    \sim \ \ &c\left(\varphi^{n+i+1}- m \left(\frac{\varphi^{i+2}}{\varphi-1}\right)\right) \\
    = \ \ &c\varphi^{i+1}\left(\varphi^n - \frac{m\varphi}{\varphi-1}\right).
\end{align}

This is eventually non-positive if
\begin{equation}
    \varphi^{n-1} - \frac{m}{\varphi-1} \ < \ 0.
\end{equation}

We solve this to obtain 

\begin{equation} \label{lowerbound1colk2}
n_M \ \geq \ \floorof{\text{log}_{\varphi} (m)-\text{log}_{\varphi}(\varphi-1)} + 1 = \floorof{\logphi{m}}+2.
\end{equation}

Hence we have equality and it is always possible to reach row $\floorof{\logphi{m}}+2$ when $m>1$.

\subsection{Solution in \texorpdfstring{$\ZZ$}{Lg}}

We will now generalise this argument to the case when $k \geq 2$, which we will later use to prove Theorem \ref{mainresult} in full generality.\newline

\noindent We need the following lemma.

\begin{lemma} \label{Sidentity}
With $S_i(n)$ defined as in \eqref{sidef}, we have
\begin{equation}
    \sum_{j=1}^{k} S_{k-j}(n)F_k(k-j+i)  \ = \  F_k(n+i+k-1).
\end{equation}
\end{lemma}
\begin{proof}

    Expanding this sum, we have
    \begin{align}
        \notag &S_0F_k(i) + S_1F_k(i+1)+\cdots + S_{k-1}F_k(i+k-1)\\ \notag
        = \ &S_{k-1}F_k(i-1) + (S_{k-1}+S_0)F_k(i) + (S_{k-1}+S_1)F_k(i+1) + \cdots \\ \label{skthing} &\cdots +(S_{k-1}+ S_{k-2})F_k(i+k-2).
    \end{align}
Now note
\begin{align}
\notag S_{k-1} + S_{\ell}  \ = \  2&F_k(n+k-2) + 2F_k(n+k-3) + \cdots \\ \notag & \cdots  + 2F_k(n+k-\ell-2) + F_k(n+k-\ell-3) + \cdots+ F_k(n-1) \\
 = \ &F_k(n+k-1)+\cdots + F_k(n+k-\ell-2). \label{skplussl}
\end{align}
We compare this to
\begin{equation}
    S_{\ell}  \ = \  F_k(n+k-2)+\cdots + F_K(n+k-\ell-2),
\end{equation}
and see that \eqref{skplussl} is indexed one higher, with one additional term. Applying this to \eqref{skthing} gives that we can replace all instances of $n$ in the expansion with $n+1$, as long as we replace all instances of $i$ with $i-1$.

We can repeatedly apply this, replacing all $n$'s with $(n+i)$'s and all $i$'s with $0$'s. Notice now that the only term that remains non-zero is $S_{k-1}F_k(i+k-1)$. By definition, $S_{k-1} = F(n+k-1)$. Then the term $F_k(n+k-1)F_k(i+k-1)$ becomes $F_k(n+i+k-1)$ under the change of $i$'s and $n$'s, as required.
\end{proof}

\begin{theorem} \label{secondfundamentaltheorem}
    In the Conway $(m,k,1)$-game, for $m>1$, it is always possible to reach the highest row theoretically attainable, namely
    \begin{equation}
        \floorof{\mathrm{log}_{\phi_{k}} (m)-\mathrm{log}_{\phi_{k}}(\phi_{k}-1)} + 1.
    \end{equation}

    In the case $m=1$, it is only possible to reach the first row in finitely many moves.
\end{theorem}

\begin{proof}

We first work for $m>1$. Along the same lines as the $k=2$ case, and inspired by Lemma \ref{nnaccijumping}, hoping to reach the $n$\textsuperscript{th} row, we want to add the following amounts of checkers to each square from row 0 down.

\begin{itemize}
    \item We want to add $S_{k-1}-m$ checkers to row $0$, to reach $S_{k-1}$ checkers.
    \item We want to add $S_{k-2}-m$ checkers to row $-1$, so that it has $S_{k-2}$ checkers.
    \item We also need to add $S_{k-1}-m$ checkers to row $-1$, so we have enough checkers to support adding $S_{k-1}-m$ to row $0$. This is a total of $S_{k-1}+S_{k-2}-2m$ checkers to be added.
    \item We need to add $S_{k-3}-m$ checkers to row $-2$.
    \item In row $-2$, we also need enough checkers to support adding the appropriate amount in the two rows above (as long as $k \geq 2$), which is $S_{k-2}+2S_{k-1}-3m$.
    \item Continue this down, in row $-i$, add the number of checkers required in that square minus $m$ ($(0 - m)$ for rows further down than $-(k-1)$, since they don't require any checkers in to reach the $n$\textsuperscript{th} row) and then add the sum of the checkers needed to be added in the $k$ rows above.
\end{itemize} 

This is shown pictorially for the first few moves in Table \ref{tab:needednnacci} below.

\begin{table}[ht]
\small
    \centering
    \begin{tabular}{|c|c|c|}
    \hline
        row & target & amount needed\\
    \hline\hline
        0 & $S_{k-1}$ & $S_{k-1} - m$ \\
    \hline
        $-1$ & $S_{k-2} + S_{k-1} -m$ & $S_{k-2} + S_{k-1} -2m$ \\
    \hline
        $-2$ & $S_{k-3} + S_{k-2} + 2S_{k-1} -3m$ & $S_{k-3} + S_{k-2} + 2S_{k-1} -4m$ \\
    \hline
        $-3$ & $S_{k-4} + S_{k-3} + 2S_{k-2} + 4S_{k-1} -7m$ & $S_{k-4} + S_{k-3} + 2S_{k-2} + 4S_{k-1} -8m$ \\
    \hline
        $-4$ & $S_{k-5} + S_{k-4} + 2S_{k-3} + 4S_{k-2} + 8S_{k-1} - 15m$ & $S_{k-5} + S_{k-4} + 2S_{k-3} + 4S_{k-2} + 8S_{k-1} - 16m$ \\
    \hline
        $\vdots$ & $\vdots$ &  $\vdots$ \\
    \end{tabular}
    \caption{Amount needed to be added to each row.}
    \label{tab:needednnacci}
\end{table}

It immediately follows that the number of checkers required to be added into row $-i$ is

\begin{equation}
\sum_{j=1}^{k} S_{k-j}(n) F_k(k-j+i) - a_im,
\end{equation}
where $a_i$ satisfies $a_i = a_{i-1} + a_{i-2} + a_{i-3} + \dots + a_{i-k} + 1$ with initial conditions $(a_0,\dots,a_{k-1}) = (1,2,4,\dots, 2^{k-1})$. From Lemma \ref{knacciplusone}, this has solution

\begin{equation}
a_i  \ = \  \sum_{j=0}^{i}F_k(k-1+j),
\end{equation}
so the total needed to be added to row $-i$ is 

\begin{equation}
\sum_{j=1}^{k} S_{k-j}(n)F_k(k-j+i) - \left( \sum_{j=0}^{i}F_k(k-1+j) \right)m.
\end{equation}
By Lemma \ref{Sidentity} we then have that the number of checkers to be added to row $-i$ is
\begin{equation}
F_k(n+i+k-1) - \left( \sum_{j=0}^{i}F_k(k-1+j) \right)m.
\end{equation}

We then seek a condition on this being eventually non-positive for consecutive rows, then no more checkers are needed on the squares to jump up to row $n$ (again note that this is in only a finite section of the column). 

Using Corollary \ref{asymptoticfk}, we have
\begin{align}
    \notag &F_k(n+i+k-1) - \left( \sum_{j=0}^{i}F_k(k-1+j) \right)m \\ \notag
    \sim \ \ &c\phi_{k}^{n+i+k-1}-cm\left(\phi_{k}^{k-1}+ \cdots + \phi_{k}^{k-1+i}\right) \\ \notag
    \sim \ \ &c\phi_{k}^{k-1}\left(\phi_{k}^{n+i}- m \left(\frac{\phi_{k}^{i+1}}{\phi_{k}-1}\right)\right) \\
    = \ \ &c\phi_{k}^{i+k-1}\left(\phi_{k}^n - \frac{m\phi_{k}}{\phi_{k}-1}\right),
\end{align}

which is eventually non-positive if
\begin{equation}
    \phi_{k}^{n-1} - \frac{m}{\phi_{k}-1} \ < \ 0.
\end{equation}

Solving this gives 

\begin{equation} \label{lowerbound1col}
n_M \ \geq \ \floorof{\text{log}_{\phi_{k}} (m)-\text{log}_{\phi_{k}}(\phi_{k}-1)} + 1.
\end{equation}

We now find an upper bound. Using the standard pagoda function and energy arguments, we have that the total energy of the board is $E_0= \frac{m\alpha^n}{1-\alpha}$, where $\alpha = 1/\phi_{k}$. We solve for $E_0 > 1$, giving

\begin{equation} \label{upperbound1col}
n_M \ \leq \ \floorof{\text{log}_{\phi_{k}}(m) - \text{log}_{\phi_{k}}(1-\alpha)}  \ = \  \floorof{\text{log}_{\phi_{k}} (m)-\text{log}_{\phi_{k}}(\phi_{k}-1)} + 1.
\end{equation}

The upper and lower bound then match, and we are done. Note that in the case $m=1$ and $k=2$ this upper bound
\begin{equation}
    n_M \ \leq \ \floorof{\logphi{m}-\logphi{\varphi-1}} + 1
\end{equation}
is strict, since both sides are integers and taking floors does not change the strictness from requiring $E_0 > 1$. Hence we have $n_M < 2$, and it is clearly possible to reach the first row, so $n_M=1$.
\end{proof}

Unfortunately, we cannot easily use the same idea in $d$-dimensions, since there is much more freedom as to what is a `good' move, and so we can't only consider moves in one direction. Instead, we will consider the $d$-dimensional game as multiple instances of the 1-dimensional game and reduce the problem to a more manageable form.

\begin{remark}
    Interestingly, if we impose the restriction that each square can never contain more than $m$ checkers, this algorithm still works, since we can stagger the moves we make (using a greedy algorithm) so as to never breach this cap.
\end{remark}

\subsection{Bounds on checkers in row 1}

In order to project into lower dimensions, we find the maximum number of checkers that can be placed on row 1.

\begin{lemma} \label{amountaddedfirstcol}
    In the Conway $(m,k,1)$-game, it is always possible to add $\floorof{\frac{m}{\phi_{k}-1}}$ checkers to the first row, where $\phi_{k}$ is the $k$-nacci constant. Further, it is never possible to add $\floorof{\frac{m}{\phi_k-1}}+1$ checkers.
\end{lemma}

\begin{proof}

We first find the upper bound. Set the target square to be on the first row, then we have that the initial energy of the board is 

\begin{equation}
    E_0  \ = \  \frac{m\alpha}{1-\alpha}  \ = \  \frac{m}{\phi_{k}-1}.
\end{equation}
The energy of the target square is 1, so we can have at most $\floorof{E_0}$ checkers in row 1.\bigskip

We now want to find a lower bound. To put $\floorof{\frac{m}{\phi_{k}-1}}$ checkers onto the first row, we require that many checkers on each of rows 0 to $-(k-1)$.

We set $M = \floorof{\frac{m}{\phi_{k}-1}}$, then have the following sequence.

\begin{itemize}
    \item Row $0$ needs $M$ checkers, so we need to add $M-m$.
    \item Row $-1$ also needs $M$ checkers, however we also need enough checkers to add $M-m$ to the first row. Hence we need to add $2M-2m$, since there are already $m$ checkers on the square.
    \item Row $-2$ again needs $M$ checkers, as well as enough checkers to support jumping into both the first and second rows. Hence we need to add $4M-4m$ checkers to this row
    \item Continue this down the board, always summing the amount needed to be added on the $k$ squares above.
\end{itemize}

From this, we get that row $-i$ requires

\begin{equation}
    F_k(k+i) M - a_im
\end{equation}
checkers, where $a_i$ satisfies $a_i = a_{i-1}+\cdots+a_{i-k}+1$. From Lemma \ref{knacciplusone} we have that number of checkers needed to be added to row $-i$ is

\begin{equation}
    F_k(k+i)M - \sum_{j=0}^{i}F_k(k-1+i)m.
\end{equation}
As usual, we want this to be eventually negative, or equal to zero. To do this, we again use Corollary \ref{asymptoticfk}.

\begin{align}
    \notag &F_k(k+i)M - \sum_{j=0}^{i}F_k(k-1+i)m \\ \notag
    \sim \ \ &c\phi_{k}^{k+1}M - c\left( \phi_{k}^{k-1}+ \cdots + \phi_{k}^{k-1+i} \right)m \\ \notag
    \sim \ \ &\phi_{k}^{k+i}M - \phi_{k}^{k-1} \frac{\phi_{k}^{i+1}}{\phi_{k}-1}m \\
    \sim \ \ &\phi_{k}^{k+i}\left(M-\frac{m}{\phi_{k}-1}\right).
\end{align}
This is eventually less than or equal to zero if $M-\frac{m}{\phi_{k}-1} < 0$. Notice

\begin{align}
    M-\frac{m}{\phi_{k}-1}  \ = \  \floorof{\frac{m}{\phi_{k}-1}} - \frac{m}{\phi_{k}-1} \ < \ 0,
\end{align}
so we are done, it is always possible to add $\floorof{\frac{m}{\phi_{k}-1}}$ checkers onto the first row. Note again that this only uses a finite section of the column. This also holds when $m=1$, since $\phi_k \geq \varphi$.
\end{proof}

\subsection{Bounds in \texorpdfstring{$d$}{Lg} dimensions}

We now have all the tools we need to prove Theorem \ref{mainresult}. We first prove the upper bound.

\begin{lemma}
    In the Conway $(m,k,d)$-game, with $\phi_{k}$ the $k$-nacci constant,
    \begin{equation}
        n_M \ \leq \ \floorof{\logPhi{m}+\logPhi{\frac{(\phi_{k}+1)^{d-1}}{(\phi_{k}-1)^d}}} + 1.
    \end{equation}
\end{lemma}

\begin{proof}
    To find the initial energy in $d$-dimensions, we sum the energies of subspaces of dimension $d-1$. In particular, the energy of the board in $d$-dimensions, $E(d)$, is
    \begin{equation}
        (1+2\alpha+2\alpha^2+ \cdots)E(d-1)  \ = \  \frac{1+\alpha}{1-\alpha}\cdot E(d-1).
    \end{equation}
    We have $E(1) = \frac{m\alpha^n}{1-\alpha}$, so
    \begin{equation} \label{Eddim}
        E(d)  \ = \  m\alpha^n\frac{(1+\alpha)^{d-1}}{(1-\alpha)^d}.
    \end{equation}
    As usual, we solve for $E(d) > 1$ and take floors to obtain the bound 
    \begin{equation}
        n_M \ \leq \ \floorof{\logPhi{m}+\logPhi{\frac{(\phi_{k}+1)^{d-1}}{(\phi_{k}-1)^d}}} + 1.
    \end{equation}
\end{proof}

Note that when we set $k=2$, $m=1$ and hence $\phi_{k} = \varphi$ this gives
\begin{equation}
    n_M \ < \ 3d-1,
\end{equation}
where we have strictness since  in this case

\begin{equation}
\logPhi{m}+\logPhi{\frac{(\phi_{k}+1)^{d-1}}{(\phi_{k}-1)^d}}
\end{equation}
is an integer. This agrees with the bounds found in \cite{ddimpaper} and \cite{desert}.\bigskip

\begin{theorem}
    In the Conway $(m,k,d)$-game, it is possible to reach the theoretically highest possible row for almost all choices of $m$.
\end{theorem}

\begin{proof}

When projecting from $\ZZ^d$ to $\ZZ^{d-1}$, it is possible to put $m + 2\floorof{\frac{m}{\phi_{k}-1}} = \frac{\phi_{k} +1}{\phi_{k}-1}m - \varepsilon_{d-2}$ checkers onto every square in the space of lower dimension, where $\varepsilon_{d-2} \in (0,2)$ is twice the error from the floor function.

Hence, from $\ZZ^d$ and projecting down to $\ZZ$, we have that on each square we can place

\begin{equation}
    \left(\frac{\phi_{k}+1}{\phi_{k}-1}\right)^{d-1}m - \sum_{i=0}^{d-2} \left(\frac{\phi_{k}+1}{\phi_{k}-1}\right)^{i}\varepsilon_i
\end{equation}
checkers. Note this is possible in finite time since all the algorithms used only require a finite section of the board.

\begin{definition} \label{Cddefgeneral}
    
We write $C(d)$ for a constant of the form

\begin{equation} 
    C(d)  \ = \  \sum_{i=0}^{d} \left(\frac{\phi_{k}+1}{\phi_{k}-1}\right)^{i}\varepsilon_i,
\end{equation}

for a collection of $\varepsilon_i \in (0,2)$.

\end{definition}

\begin{remark}
    Note that when we write $C(d)$ the collection of $\varepsilon_i$'s is implicit.
\end{remark}

\noindent Applying Theorem \ref{secondfundamentaltheorem}, we then have that the maximum attainable row is
\begin{equation}
\floorof{\logPhi{\left(\frac{\phi_{k}+1}{\phi_{k}-1}\right)^{d-1}m - C(d-2)} - \logPhi{\phi_{k}-1}}+1.
\end{equation}
We also have

\begin{align}
    \notag &\logPhi{\left(\frac{\phi_{k}+1}{\phi_{k}-1}\right)^{d-1} - \frac{C(d-2)}{m}} \\
    = \ &\logPhi{\left(\frac{\phi_{k}+1}{\phi_{k}-1}\right)^{d-1}} + \logPhi{1-\frac{C(d-2)}{m}\cdot \left(\frac{\phi_{k}-1}{\phi_{k}+1}\right)^{d-1}}.
\end{align}
We define the following function.
\begin{definition}\label{mathcalE}
Write
\begin{equation}
    \mathcal{E}(m)  \ := \  -\logPhi{1-\frac{C(d-2)}{m}\cdot \left(\frac{\phi_{k}-1}{\phi_{k}+1}\right)^{d-1}}.
\end{equation}
\end{definition}
\noindent We then have that the maximum attainable row is

\begin{equation}
    \floorof{\logPhi{m} + \logPhi{\frac{(\phi_{k}+1)^{d-1}}{(\phi_{k}-1)^d}} - \mathcal{E}(m)}+1.
\end{equation}

We now fix $k,d$ and write

\begin{equation}
A := \frac{(\phi_k+1)^{d-1}}{(\phi_k-1)^d}.
\end{equation}

Define 
\begin{align}
    S &\ := \ \{m: \text{the maximum row is unreachable in the $(m,k,d)$-game}\} \nonumber\\
    & \ \  =\  \{m : \{\logPhi{Am}\} < \mathcal{E}(m)\},
\end{align}
where $\{\logPhi{Am}\}$ denotes the fractional part of $\logPhi{Am}$.  We will show that $S$ has lower asymptotic density 0. Clearly $\mathcal{E}(m) \to 0$ as $m\to \infty$, so let $\varepsilon>0$, and there is $M$ s.t. $\mathcal{E}(m) \leq \varepsilon$ for each $m > M$. Then we have 

\begin{align}
    S & \ \subseteq \ \{1,...,M\} \ \cup \ \{m: \{\logPhi{Am}\} \leq \varepsilon\}  \nonumber \\
    & \ = \ \{1,...,M\} \ \cup \ \left(\bigcup_{c=1}^\infty \{m : c \leq \logPhi{Am} \leq c+\varepsilon\}\right) \nonumber\\
    & \ = \ \{1,...,M\} \ \cup \ \left(\bigcup_{c=1}^{\infty}\frac{\phi_k^c}{A}[1, \phi_k^\varepsilon]\right).
\end{align}

Since an interval of length $\ell$ contains at most $\ell+1$ integers, the set $\bigcup_{c=1}^{N}\frac{\phi_k^c}{A}[1, \phi_k^\varepsilon]$ contains at most $N + \sum_{c=1}^N\frac{\phi_k^c}{A}(\phi_k^\varepsilon-1)$ integers. Further, $\bigcup_{c=1}^{N}\frac{\phi_k^c}{A}[1, \phi_k^\varepsilon] \subseteq [0, \phi_k^{N+\varepsilon}]$, so the lower asymptotic density of $S$ is at most

\begin{equation}
    \lim_{N\to \infty} \frac{1}{\phi_k^{N+\varepsilon}-1}\left(M+N+ \frac{\phi_k(\phi_k^\varepsilon-1)(\phi_k^N-1)}{A(\phi_k-1)}\right) \ =\  \frac{\phi_k(\phi_k^\varepsilon-1)}{A(\phi_k-1)\phi_k^\varepsilon}.
\end{equation}

Since $A$ is constant and $\varepsilon$ arbitrary, $\frac{\phi_k(\phi_k^\varepsilon-1)}{A(\phi_k-1)\phi_k^\varepsilon} \to 0$ as $\varepsilon \to 0$, so the lower asymptotic density of $S$ is zero.
\end{proof}

An example, projecting $\ZZ^2$ onto $\ZZ$, with $k=2$, is shown here in Table \ref{tab:projection}.

\begin{table}[ht]
    \centering
    \begin{tabular}{c|c|c|c|c|c|c|c|c}
    \hline\hline
        $\cdots$ & 0 & 0 & 0 & $m+2\floorof{\varphi m}$ & 0 & 0 & 0  & $\cdots$\\
    \hline
        $\cdots$ & 0 & 0 & 0 & $m+2\floorof{\varphi m}$ & 0 & 0 & 0  & $\cdots$\\
    \hline
        $\cdots$ & 0 & 0 & 0 & $m+2\floorof{\varphi m}$ & 0 & 0 & 0  & $\cdots$\\
    \hline
    $\vdots$ & $\vdots$ & $\vdots$ & $\vdots$ & $\vdots$ & $\vdots$ & $\vdots$ & $\vdots$ & $\vdots$ \\

    \end{tabular}
    \caption{Projection from $\ZZ^2$ to $\ZZ$.}
    \label{tab:projection}
\end{table}\bigskip

We now consider the cases when the upper bound cannot be reached.

\begin{lemma}
    For any choice of $d \geq 1$ and $m,k > 1$ we have
    \begin{equation}
        \mathcal{E}(m) \ < \ 1,
    \end{equation}
    where $\mathcal{E}(m)$ is as in Definition \ref{mathcalE}. In particular, for $m>1$ it is always possible to attain within one of the upper bound.
\end{lemma}

\begin{proof}
    We have

    \begin{align}
        \notag |\mathcal{E}(m)|  \ = \  & \left\lvert\logPhi{1-\frac{C(d-2)}{m}\cdot \left(\frac{\phi_{k}-1}{\phi_{k}+1}\right)^{d-1}}\right\rvert.
    \end{align}
    We also have $|\logPhi{x}| < 1$ if and only if $x \in (1/\phi_{k}, \phi_{k})$, so we need to show
    \begin{equation}
        \frac{C(d-2)}{m}\cdot \left(\frac{\phi_{k}-1}{\phi_{k}+1}\right)^{d-1} \ \in \ \left(1-\phi_{k}, 1-\frac{1}{\phi_{k}}\right),
    \end{equation}
    or equivalently, since this quantity is positive and $\phi_{k} \in (1,2)$

    \begin{equation}
        \frac{C(d-2)}{m}\cdot \left(\frac{\phi_{k}-1}{\phi_{k}+1}\right)^{d-1} \ < \ 1-\frac{1}{\phi_{k}}.
    \end{equation}

    Note that by the definition of $C(d)$ in Definition \ref{Cddefgeneral}, we have
    \begin{equation} \label{cbound}
        C(d) \ < \ 2\frac{\left(\frac{\phi_{k}+1}{\phi_{k}-1}\right)^{d+1}-1}{\left(\frac{\phi_{k}+1}{\phi_{k}-1}\right) - 1}  \ = \  (\phi_{k}+1)\left(\left(\frac{\phi_{k}+1}{\phi_{k}-1}\right)^{d}-1\right).
    \end{equation}
    Hence we have

    \begin{align} \label{chainofineqs}
        \notag &\frac{C(d-2)}{m}\cdot \left(\frac{\phi_{k}-1}{\phi_{k}+1}\right)^{d-1} \\ \notag < \ \ &\frac{1}{2}(\phi_{k}+1)\left(\left(\frac{\phi_{k}+1}{\phi_{k}-1}\right)^{d-2}-1\right)\left(\frac{\phi_{k}-1}{\phi_{k}+1}\right)^{d-1} \\ \notag 
        < \ \ &\frac{1}{2}(\phi_{k}+1)\left(\frac{\phi_{k}-1}{\phi_{k}+1}- \left(\frac{\phi_{k}-1}{\phi_{k}+1}\right)^{d-1}\right) \\ \notag
        < \ \ &\frac{1}{2}(\phi_{k}-1)\\ < \ \ &\frac{\phi_{k}-1}{\phi_{k}},
    \end{align}
    since $\phi_{k} \in (1,2)$.
\end{proof}

\begin{remark} \label{remark4}
    
Note that in order to bound $\mathcal{E}(m)$ for $m=1$, we would need to find a better bound on $C(d)$. To show $\mathcal{E}(m) < \ell$, we would need
\begin{equation} \label{breaks}
    \frac{C(d-2)}{m}\cdot \left(\frac{\phi_{k}-1}{\phi_{k}+1}\right)^{d-1} \ < \ 1-\frac{1}{\phi_{k}^\ell}.
\end{equation}
Taking $d \to \infty$ after substituting the bound for $C(d-2)$ in a similar manner to \eqref{chainofineqs} gives an upper bound converging to $\phi_k-1$, then taking $k \to \infty$ causes the upper bound to converge to 1, which would cause \eqref{breaks} to break.

\end{remark}

In the case $k=2$, it is possible to explicitly find the candidates where we might fail to reach the upper bound.

\begin{definition}
    We define the \textit{Lucas} numbers by the recurrence
    \begin{equation}
        a_{i+2} \ = \ a_{i+1}+a_{i},
    \end{equation}
    and initial conditions $a_0=2$, $a_1=1$. We denote the $n^\mathrm{th}$ Lucas number as $L(n)$.
\end{definition}

The Lucas numbers can be computed by the formula $L(n) = \varphi^{n} + (-1/\varphi)^n$. When $n$ is even and large, $(-1/\varphi)^{n}$ is a very small positive number, however $L(n)$ is necessarily an integer, so must be very slightly above a power of $\varphi$. Hence the even-indexed Lucas numbers provide candidates for where we might fail to obtain the upper bound.

When $d=2$, the algorithm does in fact fail to achieve the upper bound for all of $L(2), L(4),$ $ \dots, L(30)$. It succeeds, however, at $L(32) = 4870847$, so does not fail at every even-indexed Lucas number.

\subsection{Maximum number of checkers in one square}

We finish with a bound on the maximum number of checkers that can be moved to occupy one square.

\begin{theorem}
    In the Conway $(m,k,d)$-game, it is possible to reach a state with at least
    \begin{equation}
        \left(\frac{\phi_{k}+1}{\phi_{k}-1}\right)^{d}m - C(d-1)
    \end{equation}
    checkers on one square. Further, it is not possible to place more than $\floorof{\left(\frac{\phi_k+1}{\phi_k-1}\right)^dm}$ checkers onto a square.
\end{theorem}

\begin{proof}
    We will first project onto $\ZZ$ in the usual way, obtaining
    \begin{equation}
        \left(\frac{\phi_{k}+1}{\phi_{k}-1}\right)^{d-1}m - C(d-2)
    \end{equation}
    checkers on each square. Since the algorithm in Lemma \ref{amountaddedfirstcol} uses only a finite section of the board, by choosing a square sufficiently far down the board it is possible to add 
    \begin{equation}
        2\floorof{\frac{m}{\phi_{k}-1}} \ = \ \frac{2m}{\phi_{k}-1} - \varepsilon
    \end{equation}
    checkers onto the square, for some $\varepsilon \in (0,2)$. We do this by using the result of Lemma \ref{amountaddedfirstcol} to add $\floorof{\frac{m}{\phi_k-1}}$ onto the square from above and below. We then have that, replacing $m$ with the number of checkers we have amassed on each square, in one square we can place 

    \begin{align}
        \notag &\left(\frac{\phi_{k}+1}{\phi_{k}-1}\right)^{d-1}m - C(d-2) + \frac{2}{\phi_{k}-1}\left(\left(\frac{\phi_{k}+1}{\phi_{k}-1}\right)^{d-1}m - C(d-2)\right) - \varepsilon \\ \notag
        = \ &\left(\frac{\phi_{k}+1}{\phi_{k}-1}\right)^{d}m - \frac{\phi_{k}+1}{\phi_{k}-1}C(d-2)-\varepsilon \\ \label{lowerboundrallying}
        = \ &\left(\frac{\phi_{k}+1}{\phi_{k}-1}\right)^{d}m - C(d-1)
    \end{align}

    checkers.\newline

    We now derive an upper bound. Choose a target square on the board that we will try to add checkers to, and weight the board in the usual way, by powers of $1/\phi_{k}$ defined by the taxicab distance from the target square. By the same argument as in \eqref{Eddim}, we have
    \begin{equation}
        E(d) \ = \ (1+2\alpha+ 2\alpha^2+\cdots)E(d-1) \ = \ \frac{1+\alpha}{1-\alpha} \cdot E(d-1).
    \end{equation}

    In this case, however, if the target square is $n$ squares down, we have 
    \begin{align}
        \notag E(1) \ = \ & \: (1+\alpha + \alpha^2+ \cdots)m + (\alpha+\alpha^2 + \cdots + \alpha^n)m \\
        = \ &\left(\frac{1}{1-\alpha}+ \frac{\alpha(1-\alpha^n)}{1-\alpha}\right)m \ < \ \left(\frac{1+\alpha}{1-\alpha}\right)m.
    \end{align}
    Hence the energy in $d$-dimensions is less than 
    \begin{equation}
        \left(\frac{1+\alpha}{1-\alpha}\right)^{d}m \ = \ \left(\frac{\phi_{k}+1}{\phi_{k}-1}\right)^{d}m,
    \end{equation}
    so it is possible to place at most 
    \begin{equation}
       \floorof{\left(\frac{\phi_{k}+1}{\phi_{k}-1}\right)^{d}m}
    \end{equation}
    checkers onto a single square. Note this asymptotically matches the lower bound obtained in \eqref{lowerboundrallying} as $m$ grows.
\end{proof}

\section{Acknowledgements}

We thank participants of the 21\textsuperscript{st} International Fibonacci Conference for helpful discussions.

\noindent This work was partially supported by NSF grants DMS2241623 and DMS2422706, as well as Emmanuel College Cambridge, Princeton University and Williams College. We also thank the Herschel-Smith Fellowship for their support. Finally, we thank the referees for many helpful comments, which have much improved our paper.

\printbibliography

\end{document}